\documentclass[10pt, letterpaper, reqno]{amsart}

\usepackage[a4paper,left=38mm,right=38mm,top=25mm,bottom=25mm,marginpar=25mm]{geometry}

\usepackage{amsmath, amsthm, amsfonts, amssymb, amscd, mathtools, bm} 
\usepackage{xcolor} 
\usepackage[unicode=true]{hyperref} 
\usepackage[shortlabels]{enumitem} 
\usepackage{empheq}
\usepackage[T1]{fontenc}

\theoremstyle{plain}
\newtheorem{theorem}{Theorem}[section]

\newtheorem{proposition}[theorem]{Proposition}

\theoremstyle{definition}

\newtheorem{example}[theorem]{Example}

\theoremstyle{remark}

\numberwithin{equation}{section}

\newcommand{\ap}{\alpha_p}

\begin{document}

\baselineskip=17pt

\title[]{Relation between asymptotic $L_p$-convergence and some classical modes of convergence}

\author[Nuno J. Alves]{Nuno J. Alves}
\address[Nuno J. Alves]{
      University of Vienna, Faculty of Mathematics, Oskar-Morgenstern-Platz 1, 1090 Vienna, Austria.}
\email{nuno.januario.alves@univie.ac.at}

\author[Giorgi G. Oniani]{Giorgi G. Oniani}
\address[Giorgi G. Oniani]{
      Kutaisi International University, School of Computer Science and Mathematics, Kutaisi 4600, Georgia.}
\email{giorgi.oniani@kiu.edu.ge}

\begin{abstract} 
Asymptotic $L_p$-convergence, which resembles convergence in $L_p$, was introduced to address a question in diffusive relaxation. This note aims to compare asymptotic $L_p$-convergence with convergence in measure and in weak $L_p$ spaces. One of the results characterizes convergence in measure on finite measure spaces in terms of asymptotic $L_p$-convergence.

\end{abstract}



\maketitle
\thispagestyle{empty}

Let $(X, \mathbb{X}, \mu)$ be a measure space. By a measurable function, we understand a real-valued $\mu$-measurable function defined on $X$. The set of all such measurable functions (identified $\mu$-a.e.) is denoted by $M(X)$. We use the notation $A^c$ for the complement of a set $A \subseteq X$ with respect to $X$. It will be assumed that $p$ is a number belonging to $[1,\infty)$.\par 
The main purpose of this note is to compare  the convergences in measure and in weak $L_p$ spaces with asymptotic $L_p$-convergence. The latter concept was motivated by a question on convergence in diffusive relaxation and introduced in \cite{alves2024mode}. Specifically, a sequence $(f_n)$ of measurable functions is said to \textit{asymptotically $L_p$-converge} (in short, $\ap$-\textit{converge}) to a measurable function $f$ if there exists a sequence of measurable sets $(B_n)$ with $\mu (B_n^c) \to 0$ as $n \to \infty$ such that $\int_{B_n} |f_n - f|^p \, \mathrm{d}\mu \to 0$ as $n \to \infty.$

\section{Asymptotic $L_p$-convergence vs. convergence in measure}

 In \cite{alves2024mode}, it was shown that $\ap$-convergence implies convergence in measure, and an example was given of a sequence of functions that converges in measure but does not $\ap$-converge. This example, however, is in an infinite measure space. For finite measure spaces, these two notions of convergence are, in fact, equivalent. Recall that a sequence $(f_n)$ is said to \textit{converge in measure} to a measurable function $f$ if for every $\delta>0$, $\mu\big(  \{x \in X \ | \ |f_n(x) - f(x)| \geq \delta \} \big)\to 0$ as $n \to \infty$.
We prove the following:
\begin{theorem} \label{alpha_mu}
Let  $f_n$ $(n\in \mathbb{N})$ and $f$ be measurable functions. If $(f_n)$ $\ap$-converges to $f$, then $(f_n)$ converges to $f$ in measure. On the other hand, if $\mu(X) < \infty$ and $(f_n)$ converges to $f$ in measure, then $(f_n)$ $\ap$-converges to $f$.
\end{theorem}
The first part of this theorem was established in \cite{alves2024mode}, but here we provide a simpler proof. 
\begin{proof}
First, assume that $(f_n)$ $\ap$-converges to $f$. For each $\delta > 0$ let the set $E_n(\delta)$ be given by
\[E_n(\delta) = \{x \in X \ | \ |f_n(x) - f(x)| \geq \delta \}. \]
Then 
\begin{align*}
\delta^p \mu(E_n(\delta)) & = \delta^p \mu(E_n(\delta) \cap B_n) +  \delta^p \mu(E_n(\delta) \cap B_n^c) \\
& \leq \int_{E_n(\delta) \cap B_n} \delta^p \, \mathrm{d}\mu + \delta^p \mu(B_n^c) \\
& \leq \int_{B_n} |f_n - f|^p \, \mathrm{d}\mu + \delta^p \mu(B_n^c)
\end{align*}
where $(B_n)$ is a sequence of measurable sets associated with the $\ap$-convergence of $(f_n)$ towards $f$. Letting $n \to \infty$ yields the desired conclusion. \par 
Now, assume that $X$ has finite measure and that $(f_n)$ converges to $f$ in measure. Then, for each $n \in \mathbb{N}$, there exists $N_n \in \mathbb{N}$ such that  $\mu(E_k(1/n)) <1/n$ whenever $k \geq N_n$. Assume, without loss of generality, that $N_{n+1} > N_n$ for every natural $n$. Let $(\lambda_k)$ be a sequence of positive numbers such that for every $n\in \mathbb{N}$, $\lambda_k = 1/n$ whenever $k \in [N_n, N_{n+1})$, and set $B_k = E_k(\lambda_k)^c$ $(k\in \mathbb{N})$. Notice that $\mu(B_k^c) \to 0$ as $k \to \infty$ and 
\[\int_{B_k} |f_k - f|^p \, \mathrm{d}\mu \leq \int_{B_k} \lambda_k^p \, \mathrm{d}\mu \leq \lambda_k^p \mu(X) \to 0 \quad \text{as} \ k \to \infty \]
which finishes the proof.
\end{proof}
Furthermore, there are notions of Cauchy sequences related to $\ap$-convergence and convergence in measure. We say that a sequence $(f_n)$ of measurable functions is \textit{asymptotically $L_p$-Cauchy} (in short, $\ap$-\textit{Cauchy}) if there exists a sequence of measurable sets $(B_n)$ with $\mu (B_n^c) \to 0$ as $n \to \infty$ such that 
\[\int_{B_n \cap B_m} |f_n - f_m|^p \, \mathrm{d}\mu \to 0 \quad  \text{as} \ n,m \to \infty, \]
and it is \textit{Cauchy in measure }if 
\[ \forall \delta > 0 \quad \mu\big(  \{x \in X \ | \ |f_n(x) - f_m(x)| \geq \delta \}  \big) \to 0  \quad \text{as} \ n,m \to \infty.\]  
It is known that $\ap$-Cauchy sequences are Cauchy in measure, and that $\ap$-Cauchy sequences (respectively, Cauchy in measure) $\ap$-converge (respectively, converge in measure) to a measurable function $f$ (see \cite{alves2024mode, bartle1995elements}). This, together with Theorem \ref{alpha_mu}, yields that in a finite measure space these two notions of Cauchy sequences are equivalent.
\begin{theorem} \label{alpha_mu_cauchy}
Let $(f_n)$ be a sequence of measurable functions. If $(f_n)$ is $\ap$-Cauchy, then $(f_n)$ is Cauchy in measure. On the other hand, if $\mu(X) < \infty$ and $(f_n)$ is Cauchy in measure, then $(f_n)$ is $\ap$-Cauchy.
\end{theorem}
\begin{proof}
The first part can be deduced similarly to the proof of the first part of Theorem \ref{alpha_mu}, but with $E_n(\delta)$ replaced by $E_{n,m}(\delta) = \{ x \in X \ | \ |f_n(x) - f_m(x)| \geq \delta\}$ and $B_n$ replaced by $B_n \cap B_m$. \par 
Regarding the second part, assume that $\mu(X) < \infty$ and that $(f_n)$ is Cauchy in measure. Then, there exists a measurable function $f$ such that $(f_n)$ converges to $f$ in measure. By Theorem \ref{alpha_mu} it follows that $(f_n)$ $\ap$-converges to $f$, and hence $(f_n)$ is $\ap$-Cauchy, as desired.  
\end{proof}

\section{Asymptotic $L_p$-convergence vs. convergence in weak $L_p$ spaces}
The \textit{weak $L_p$ space}, denoted by $L_{p,\infty}$, consists of all measurable functions $f$ such that
\[\sup_{\delta > 0} \, \delta^p \mu(\{ x \in X \ | \ |f(x)| \geq \delta \})  < \infty. \]
Let  $f_n$ $(n\in \mathbb{N})$ and $f$ be functions belonging to $L_{p,\infty}$. The sequence $(f_n)$  is said to \textit{converge in} $L_{p,\infty}$ to  $f$ (see, e.g., \cite{castillolp}) if \[\sup_{\delta > 0} \, \delta^p \mu(\{ x \in X \ | \ |f_n(x)-f(x)| \geq \delta \}) \to 0 \quad \text{as} \ n \to \infty.\]
It is clear that this convergence implies convergence in measure. \par

 The next example shows that asymptotic $L_p$-convergence does not imply convergence in weak $L_p$ spaces even for spaces with finite measure.

\begin{example} \label{example1}
Let $X = [0,1]$, $\mathbb{X}$ be the collection of all Lebesgue measurable subsets of $[0,1]$, and $\mu$ be  Lebesgue measure on $\mathbb{X}$. Consider the sequence of functions $(f_n)$ defined by $f_n = n^{1/p} \chi_{[0,1/n]}$. Then $(f_n)$ $\ap$-converges to $0$ (zero function), but it does not converge to $0$ in $L_{p,\infty}$. To check the latter, for each $\delta > 0$ let $F_n(\delta) = \delta^p \mu( \{x \in [0,1] \ | \ |f_n(x)| \geq \delta \})$ and notice that 
\begin{equation*}
F_n(\delta) = \begin{dcases}
\frac{\delta^p}{n} \quad & \text{if} \ 0 < \delta \leq n^{1/p}, \\
0 \quad & \text{if} \  \delta > n^{1/p}.
\end{dcases}
\end{equation*}
Then $\sup_{\delta > 0} F_n(\delta) = 1$ for every natural $n$ and hence $(f_n)$ does not converge to $0$ in $L_{p,\infty}$.
\end{example}

The next example shows that the inverse implication is also false, i.e., convergence in weak $L_p$ spaces does not imply asymptotic $L_p$-convergence. 

\begin{example} \label{example2}
Let $X = [1,\infty)$, $\mathbb{X}$ be the collection of all Lebesgue measurable subsets of $[1,\infty)$, and $\mu$ be  Lebesgue measure on $\mathbb{X}$. Let $f_n$ be given by \[f_n(x) = \frac{1}{(nx)^{\frac{1}{p}}}  \;\;\;\;(x\in[1,\infty)).\]
Then 
\begin{align*}
F_n(\delta) & = \delta^p \mu( \{x \in[1,\infty) \ | \ |f_n(x)| \geq \delta \}) \\
&  = \begin{dcases}
 \frac{1-n \delta^p}{n}  \quad & \text{if} \ 0 < \delta \leq \frac{1}{n^{1/p}}, \\
0 \quad & \text{if} \  \delta > \frac{1}{n^{1/p}},
\end{dcases}
\end{align*}
and so $\sup_{\delta > 0} F_n(\delta) = 1/n \to 0$ as $n \to \infty$, i.e., $(f_n)$ converges to $0$ in $L_{p,\infty}$. \par 
Next, we show that $(f_n)$ does not $\ap$-converge to $0$. Let $(B_n)$ be any sequence of measurable subsets of $[1,\infty)$ such that $\mu(B_n^c) \to 0$ as $n \to \infty$. Let $N \in \mathbb{N}$ be such that $\mu(B_n^c) < 1$ for all $n \geq N$. Then, for $n \geq N$, 
\[\frac{1}{n} \int_{B_n^c} \frac{1}{x} \, \mathrm{d}x \leq \frac{1}{n} \mu(B_n^c) < \frac{1}{n} < \infty. \]
Notice that \[\frac{1}{n} \int_{B_n} \frac{1}{x} \, \mathrm{d}x = \frac{1}{n} \int_1^\infty \frac{1}{x}  \, \mathrm{d}x - \frac{1}{n} \int_{B_n^c} \frac{1}{x}  \, \mathrm{d}x. \]
Since the first integral on the right-hand side is infinite and for $n \geq N$ the second integral on the right-hand side is finite, it follows that for $n \geq N$ the integral on the left-hand side is infinite, i.e., $(f_n)$ does not $\ap$-converge to $0$.
\end{example}

Note that by Theorem 1.1, for spaces with finite measure, convergence in weak $L_p$ spaces implies asymptotic $L_p$-convergence. 

\section{Almost $L_p$ spaces vs. weak $L_p$ spaces}

Denote by $A_p$ be the space of measurable functions that are \textit{almost in} $L_p$, i.e.,
\[A_p = \left\lbrace f \in M(X) \ | \ \forall \delta > 0 \ \exists E_\delta \ \text{with} \ \mu(E_\delta) < \delta \ \text{such that} \ \int_{E_\delta^c} |f|^p \, \mathrm{d}\mu < \infty \right\rbrace. \]
These spaces are named as \textit{almost} $L_p$ \textit{spaces} and   were introduced in \cite{bravo2012optimal,calabuig2019representation}. It is clear that $L_p \subseteq A_p$, and it is not hard to see that there are functions belonging to $A_p$ but not to $L_p$.

Below, we give examples that show that in general measure spaces, $A_p \setminus L_{p,\infty}$ and $L_{p,\infty} \setminus A_p$ are non-empty. Then, at last, we prove that $L_{p,\infty}$ is contained in $A_p$ when $\mu(X) < \infty$.

\begin{example} \label{example3}
Let $X = (0,1)$, $\mathbb{X}$ be the collection of all Lebesgue measurable subsets of $(0,1)$, and $\mu$ be  Lebesgue measure on $\mathbb{X}$.	
	Let  $f : (0,1) \to \mathbb{R}$ be defined by 
	\[f(x) = \frac{1}{x^2} \quad (x\in (0,1)). \]
	
	This function does not belong to $L_{p,\infty}$. Indeed, for every $\delta\geq 1$ we have that
 $$\delta^p \mu (\{x \in (0,1) \ | \ |f(x)| 		
		 \geq \delta \}) =	\delta^{p-1/2}.
		 $$
	 Letting $\delta \to \infty$ yields the desired conclusion. 
	\par 
On the other hand, since $f$ is decreasing, it is easy to see that $f$ belongs to $A_p$. 
\end{example}

\begin{example} \label{example4}
Let $X = [1,\infty)$, $\mathbb{X}$ be the collection of all Lebesgue measurable subsets of $[1,\infty)$, and $\mu$ be  Lebesgue measure on $\mathbb{X}$. Let $f$ be defined by \[f(x) = \frac{1}{x^{\frac{1}{p}}}  \;\;\;\;(x\in [1,\infty)).\]

Using a similar reasoning as in Example \ref{example2} one can conclude that $f$ does not belong to $A_p$ but it belongs to $L_{p,\infty}$.
\end{example}

\begin{proposition}
If $\mu(X) < \infty$, then $L_{p,\infty} \subseteq A_p$. 
\end{proposition}
\begin{proof}
By hypothesis we have 
\[\sup_{\delta > 0} \delta^p \mu (\{x \in X \ | \ |f(x)|  \geq \delta \}) = C < \infty. \]
In particular, for each $k \in \mathbb{N}$ it holds that
\[ \mu (\{x \in X \ | \ |f(x)|  \geq k \}) \leq \frac{C}{k^p}.\]
Given $\delta > 0$, let $K \in \mathbb{N}$ be such that 
\[ \mu (\{x \in X \ | \ |f(x)|  \geq K \}) \leq \frac{C}{K^p} < \delta\]
and set $E_\delta = \{x \in X \ | \ |f(x)|  \geq K \}$. Then 
\[\int_{E_\delta^c} |f|^p \, \mathrm{d}\mu \leq K \mu(X) < \infty \]
which concludes the proof.
\end{proof}

\section*{Acknowledgements}
Some of the questions addressed in this note were proposed in the problem session of the $46^{\text{th}}$ Summer Symposium in Real Analysis, which took place at the University of \L\'{o}d\'{z}, Poland, in June 2024. This research was partially funded by the Austrian Science Fund (FWF) project 10.55776/F65.


\end{document}